\documentclass{article}

\usepackage{amssymb,amsmath,color,enumerate,ascmac,latexsym,diagbox}
\usepackage[all]{xy}
\usepackage[abbrev]{amsrefs}
\usepackage[mathscr]{eucal}
\usepackage[top=30truemm,bottom=30truemm,left=30truemm,right=30truemm]{geometry}
\usepackage{mathrsfs}
\usepackage{amsthm}

\newcommand{\zetp}{\mathbb{Z}_{(p)}}

\newcommand{\Q}{\mathbb{Q}}

\newcommand{\Z}{\mathbb{Z}}

\newcommand{\C}{\mathbb{C}}

\newcommand{\five}{~~~~~}

\newcommand{\fift}{~~~~~~~~~~~~~~~}

\newcommand{\mm}[1]{\mathop{\mathrm{#1}}}

\newcommand{\be}[2]{\beta_{#1}^{(#2)}}
\newcommand{\ta}[1]{\tanh^{#1}{(t/2)}}

\newcommand{\co}[2]{D_{#1}^{(#2)}}

\newcommand{\copoly}[4]{\!^{#1}\!D_{#2}^{(#3)}(#4)}
\newcommand{\copasym}[4]{\!^{#1}{\mathfrak{D}}_{#2}^{(#3)}(#4)}
\newcommand{\cosym}[3]{{\mathfrak{D}}_{#1}^{(#2)}(#3)}
\newcommand{\stirf}[2]{\left[#1 \atop #2\right]}
\newcommand{\stirs}[2]{\left\{#1 \atop #2\right\}}

\newcommand{\dnkij}{d_{n}^{(k)}(i,j)}
\newcommand{\ordp}{{\mm{ord}}_{p}}

\newcommand{\dv}[2]{\frac{d^{#1}}{{d#2}^{#1}}}
\newcommand{\pdv}[2]{\frac{{\partial}^{#1}}{\partial {#2}^{#1}}}

\usepackage{graphicx}
\usepackage{stmaryrd}
\makeatletter
\def\mapstofill@{%
	\arrowfill@{\mapstochar\relbar}\relbar\rightarrow}
\newcommand*\xmapsto[2][]{%
	\ext@arrow 0395\mapstofill@{#1}{#2}}

\usepackage[dvipdfmx]{hyperref}

\usepackage{xcolor}
\hypersetup{
	colorlinks=true,
	citebordercolor=green,
	linkbordercolor=red,
	urlbordercolor=cyan,
}

\newtheorem{theorem}{Theorem}[section]
\newtheorem{definition}[theorem]{Definition}
\newtheorem{lemma}[theorem]{Lemma}
\newtheorem{proposition}[theorem]{Proposition}
\newtheorem{corollary}[theorem]{Corollary}
\theoremstyle{definition}
\newtheorem{remark}[theorem]{Remark}
\newtheorem{example}[theorem]{Example}

\makeatletter
\def\underbrace@#1#2{\vtop {\m@th \ialign {##\crcr $\hfil #1{#2}\hfil $\crcr \noalign {\kern 3\p@ \nointerlineskip }\upbracefill \crcr \noalign {\kern 3\p@ }}}}
\def\overbrace@#1#2{\vbox {\m@th \ialign {##\crcr \noalign {\kern 3\p@ }\downbracefill \crcr \noalign {\kern 3\p@ \nointerlineskip }$\hfil #1 {#2}\hfil $\crcr }}}

\def\underbrace#1{%
	\mathop{\mathchoice{\underbrace@{\displaystyle}{#1}}
		{\underbrace@{\textstyle}{#1}}
		{\underbrace@{\scriptstyle}{#1}}
		{\underbrace@{\scriptscriptstyle}{#1}}}\limits
}
\def\overbrace#1{%
	\mathop{\mathchoice{\overbrace@{\displaystyle}{#1}}
		{\overbrace@{\textstyle}{#1}}
		{\overbrace@{\scriptstyle}{#1}}
		{\overbrace@{\scriptscriptstyle}{#1}}}\limits
}

\makeatother

\allowdisplaybreaks

\begin{document}
	
	\title{On some properties of polycosecant numbers and polycotangent numbers}
	\author{Kyosuke Nishibiro}
	\date{}
	\maketitle
	
	\setcounter{section}{0}
	
	\begin{abstract}
		Polycosecant numbers and polycotangent numbers are introduced as level two analogues of poly-Bernoulli numbers. It is shown that polycosecant numbers and polycotangent numbers satisfy many formulas similar to those of poly-Bernoulli numbers. However, there is much unknown about polycotangent numbers. For example, the zeta function interpolating them at non-positive integers has not yet been constructed. In this paper, we show some algebraic properties of polycosecant numbers and polycotangent numbers. Also, we generalize duality formulas for polycosecant numbers which essentially include those for polycotangent numbers.
	\end{abstract}
	
	\section{Introduction.}
	
	For $k\in\Z$, two types of poly-Bernoulli numbers $\{B_n^{(k)}\}$ and $\{C_n^{(k)}\}$ are defined by Kaneko as follows:
	\begin{align}
		\frac{\mm{Li}_k(1-e^{-t})}{1-e^{-t}}&=\sum_{n=0}^{\infty}B_{n}^{(k)}\frac{t^n}{n!},\\
		\frac{\mm{Li}_k(1-e^{-t})}{e^t-1}&=\sum_{n=0}^{\infty}C_{n}^{(k)}\frac{t^n}{n!},
	\end{align}
	where 
	\begin{align*}
		{\mm{Li}}_k(z)=\sum_{n=1}^{\infty}\frac{z^n}{n^k}~~(|z|<1)
	\end{align*}
	is the polylogarithm function~(see Kaneko\cite{K1}, Arakawa-Kaneko\cite{AK1}, and Arakawa-Ibukiyama-Kaneko\cite{AIK}). Since ${\mm{Li}}_1(z)=-\log{(1-z)}$, $C_n^{(1)}=(-1)^nB_n^{(1)}$ coincides with the ordinary Bernoulli number $B_n$.
	
	In \cite{KPT}, as the level two analogue of poly-Bernoulli numbers, polycosecant numbers $\{\co{n}{k}\}$ are defined by Kaneko, Pallewatta and Tsumura as follows:
	\begin{align}
		\frac{\mm{A}_{k}(\ta{})}{\sinh{t}} &= \sum_{n=0}^{\infty}\co{n}{k} \frac{t^n}{n!}, \label{genecosecant}
	\end{align}
	where
	\begin{align*}
		{\mm{A}}_k(z)=2\sum_{n=0}^{\infty}\frac{z^{2n+1}}{(2n+1)^k}={\mm{Li}}_k(z)-{\mm{Li}}_k(-z)~~(|z|<1)
	\end{align*}
	is the polylogarithm function of level two. Note that Sasaki first considered \eqref{genecosecant} in \cite{Sasaki}. Since ${\mm{A}}_1(z)=2\tanh^{-1}(z)$, $\co{n}{1}$ coincides with the ordinary cosecant number $D_n$ (see \cite{No}). Also, as its relatives, polycotangent numbers $\{\be{n}{k}\}$ are defined by Kaneko, Komori and Tsumura as follows:
	\begin{align}
		\frac{\mm{A}_{k}(\ta{})}{\tanh{t}} &= \sum_{n=0}^{\infty}\be{n}{k}\frac{t^n}{n!}.
	\end{align}
	
	\begin{remark}
		Because each generating function is an even function, we have $\co{2n-1}{k}=0$ and $\be{2n-1}{k}=0$ for $n\in\Z_{\geq1}$ and $k\in\Z$.
	\end{remark}
	
	A calculation shows that
	\begin{align}
		\be{2n}{k}&=\sum_{i=0}^{n} \dbinom{2n}{2i}D_{2i}^{(k)}\label{cotacose},\\
		\co{2n}{k}&=\sum_{i=0}^{n} \dbinom{2n}{2i}E_{2n-2i}\be{2i}{k}
	\end{align}
	hold, where $E_n$ is the Euler number defined by
	\[
	\frac{1}{\cosh{t}}=\sum_{n=0}^{\infty} E_n\frac{t^n}{n!}.
	\]  
	
	It is known that the Bernoulli numbers satisfy the following congruence relations, which are called Kummer congruences. Here, $\varphi$ is the Euler's totient function.
	
	\begin{theorem}
		Let $p$ be an odd prime number. For $m,n$ and $N\in\Z_{\geq1}$ with $m\equiv n \bmod \varphi(p^N)$ and $(p-1)\nmid n$, we have
		\[
		(1-p^{m-1})\frac{B_{m}}{m}\equiv(1-p^{n-1})\frac{B_{n}}{n} \bmod{p^N}.
		\]
	\end{theorem}

	Sakata proved Kummer type congruences for poly-Bernoulli numbers.
	
	\begin{theorem}[{\cite[Theorem 6.1]{Sa}\label{kusakata}}]
		Let $p$ be an odd prime number. For $m,n,$ and $N\in\Z_{\geq1}$ with $m\equiv n \bmod \varphi(p^N)$ and $m, n\geq N$, we have
		\begin{align*}
			B_n^{(-k)}&\equiv B_m^{(-k)} \bmod{p^N},\\
			C_{n}^{(-k)}&\equiv C_{m}^{(-k)} \bmod{p^N}.
		\end{align*}
	\end{theorem}

    \begin{remark}
    	
    	Kitahara generalized Theorem \ref{kusakata} by using $p$-adic distributions (\cite[Theorem 7]{Ki}), and Katagiri proved the congruence relation for multi-poly-Bernoulli numbers (\cite[Theorem 1.6]{Ka}).
    \end{remark}

    Also, Sakata proved the following congruence relation.
    
    \begin{theorem}[{\cite[Theorem 6.10]{Sa}}]
    	Let $p$ be an odd prime number. For $k \in\Z_{\geq 0}$ and $n, N\in\Z_{\geq 1}$ with $n\geq N$, we have
    	\begin{align}
    		\sum_{i=0}^{\varphi(p^N)-1} B_{n}^{(-k-i)}\equiv0 \bmod{p^N}. \label{sum}
    	\end{align} 
    \end{theorem}
	
	By using the explicit formula for polycosecant numbers, Pallewatta showed that polycosecant numbers satisfy the same type of congruence.
	
	\begin{theorem}[{\cite[Theorem 3.12]{Pa}\label{kupa}}]
		Let $p$ be a prime number. For $k,m,n$ and $N\in\Z_{\geq 1}$ with $2m\equiv 2n \bmod \varphi(p^N)$ and $2m, 2n\geq N$, we have
		\[
		\co{2m}{-2k+1}\equiv\co{2n}{-2k+1}\bmod{p^N}.
		\]
	\end{theorem}

    Also, Kaneko proved that poly-Bernoulli numbers satisfy duality formulas as follows:
    
    \begin{theorem}[{\cite[Theorem 2]{K1}}]
    	For $l, m\in\Z_{\geq0}$, we have
    	\begin{align}
    		B_m^{(-l)}&=B_l^{(-m)},\label{dualb}\\
    		C_m^{(-l-1)}&=C_l^{(-m-1)}.\label{dualc}
    	\end{align}
    \end{theorem}

    As its analogue, it was proved that $\co{2m}{-2l-1}$ and $\be{2m}{-2l}$ satisfy duality formulas (see \cite[Theorem 3.1]{Pa} and \cite{KKT}).
    
    \begin{theorem}
    	For $l, m\in\Z_{\geq0}$, we have
    	\begin{align}
    		\co{2m}{-2l-1}&=\co{2l}{-2m-1},\label{dualcose}\\
    		\be{2m}{-2l}&=\be{2l}{-2m}.\label{dualcota}
    	\end{align}
    \end{theorem}
	
	In Section 2, we review some properties of polycosecant numbers and polycotangent numbers. In Section 3, we prove some algebraic properties of $\co{n}{k}$ and $\be{n}{k}$. In Section 4, we define symmetrized polycosecant numbers and prove their duality formulas.
	
	\section{Preliminaries}
	
	In this section, we review some properties of $\co{n}{k}$. First of all, we prepare some notations.
	
	\begin{definition}[{\cite[Definition 2.2]{AIK}}]
		For $n,m\in\Z$, the Stirling number of the first kind $\stirf{n}{m}$ and of the second kind $\stirs{n}{m}$ are defined by the following recursion formulas, respectively:
		\begin{align*}
			\stirf{0}{0}&=1, \stirf{n}{0}=\stirf{0}{m}=0,\\
			\stirf{n}{m}&=\stirf{n-1}{m-1}+n\stirf{n}{m-1},\\
			\stirs{0}{0}&=1, \stirs{n}{0}=\stirs{0}{m}=0,\\
			\stirs{n}{m}&=\stirs{n-1}{m-1}+m\stirs{n-1}{m}.
		\end{align*}
	\end{definition}
	
	\begin{lemma}[{\cite[Proposition 2.6]{AIK}\label{expstir}}]
		For $j\in\Z_{\geq0}$, we have
		\[
		\frac{(e^t-1)^m}{m!}=\sum_{n=m}^{\infty} \stirs{n}{m}\frac{t^n}{n!}.
		\]
	\end{lemma}

    \begin{lemma}[{\cite[Proposition 2.6 (6)]{AIK}\label{bekistir}}]
    	For $m, n\in\Z_{\geq0}$, we have
    	\[
    	\stirs{n}{m}=\frac{(-1)^m}{m!}\sum_{l=0}^{m}(-1)^l\binom{m}{l}l^n.
    	\]
    \end{lemma}
	
	\begin{lemma}[{\cite[Section 3]{Pa}\label{tax2}}]
		For $m\in\Z_{\geq0}$, we have
		\[
		\ta{m}=(-1)^m\sum_{n=m}^{\infty}\sum_{j=m}^{n} (-1)^j\frac{j!}{2^j}\binom{j-1}{m-1}\stirs{n}{j}\frac{t^n}{n!}.
		\]
	\end{lemma}

    Here, we recall some properties of polycosecant and polycotangent numbers.
	
	\begin{proposition}[{\cite[Theorem 3.4]{Pa}\label{explicose}}]
		For $n\in\Z_{\geq0}$ and $k\in\Z$, we have
		\begin{align}
			\co{n}{k}=\sum_{i=0}^{\lfloor n/2 \rfloor}\frac{1}{(2i+1)^{k+1}}\sum_{j=2i+1}^{n+1}\frac{(-1)^{j+1}j!}{2^{j-1}}\binom{j-1}{2i}\stirs{n+1}{j}.
		\end{align}
	\end{proposition}

    The following expressions \eqref{eqsasaki} were noticed by Sasaki.

    	\begin{proposition}\label{sasaki}
    	For $n, k\in\Z_{\geq0}$, we have
    	\begin{align}
    		\co{2n}{-k}=\sum_{i=1}^{\min{\{2n+1,k\}}}\frac{i!(i-1)!}{2^{i-1}}\stirs{k}{i}\stirs{2n+1}{i}.\label{eqsasaki}
    	\end{align}
    \end{proposition}
	
	\begin{proposition}[{\cite[Proposition 3.3]{Pa}\label{recurcose}}]
		For $n\in\Z_{\geq0}$ and $k\in\Z$, we have
		\[
		\co{n}{k-1}=\sum_{m=0}^{\lfloor n/2\rfloor} \binom{n+1}{2m+1} \co{n-2m}{k}.
		\]
	\end{proposition}
	
	\begin{lemma}[{\cite[Section 3]{Pa}\label{genfuncose}}]
		Let $f(t,y)$ be the generating function of $\co{n}{-k}$ such as
		\[
		f(t,y) = \sum_{n=0}^{\infty}\sum_{k=0}^{\infty} \co{n}{-k}\frac{t^n}{n!}\frac{y^k}{k!}.
		\]
		Then we have
		\[
		f(t,y)=1+\frac{e^t(e^y-1)}{1+e^t+e^y-e^{t+y}}+\frac{e^{-t}(e^y-1)}{1+e^{-t}+e^y-e^{-t+y}}.
		\]
	\end{lemma}
	
	It is known that poly-Bernoulli numbers satisfy the following recurrence formulas.
	
	\begin{proposition}[{\cite[Theorem 3.2]{OS2}}]
		For $k\in\Z$ and $m, n\in\Z_{\geq1}$ with $m\geq n$, we have
		\begin{align*}
			\sum_{m\geq j\geq i\geq0}^{}&(-1)^i\stirf{m+2}{i+1}B_{n}^{(-k-j)}=0,\\
			\sum_{j=0}^{m}&(-1)^j\stirf{m+1}{j+1}C_{n-1}^{(-k-j)}=0.
		\end{align*}
	\end{proposition}

    Also, there are some studies to determine the denominator of poly-Bernoulli numbers. For an odd prime number $p$, we denote by $\ordp$ the $p$-adic valuation and let
    \[
    \zetp=\left\{\frac{a}{b}\in\Q~ \middle| ~\ordp(b)=0\right\}.
    \]
    
    The Clausen von-Staudt theorem is as follows.
    
    \begin{theorem}[{\cite[Theorem 3.1]{AIK}\label{theoremvon}}]
    	Let $n$ be 1 or an even positive integer. Then we have
    	\[
    	B_n+\sum_{\substack{p:\mm{prime}\\ (p-1)\mid n}}\frac{1}{p} \in\Z.
    	\]
    \end{theorem}
    
    This can be generalized to that for poly-Bernoulli numbers as follows.
    
    \begin{theorem}[{\cite[Theorem 14.7]{AIK}\label{clausenpb}}]
    	Let $p$ be a prime number. For $k\in\Z_{\geq 2}$ and $n\in\Z_{\geq1}$ satisfying $k+2\leq p\leq n+1$, we have following results.
    	\begin{enumerate}
    		\item When $(p-1) \mid n$, $p^kB_n^{(k)}\in\zetp$ and
    		\[
    		p^kB_n^{(k)}\equiv-1 \bmod p.
    		\]
    		\item When $(p-1) \nmid n$, $p^{k-1}B_n^{(k)}\in\zetp$ and
    		\[
    		p^{k-1}B_n^{(k)} \equiv \begin{cases}
    			\displaystyle \frac{1}{p}\stirs{n}{p-1}-\frac{n}{2^k} \bmod{p}& (n\equiv1\bmod{p-1}), \\
    			\displaystyle \frac{(-1)^{n-1}}{p}\stirs{n}{p-1} \bmod{p} & (otherwise).
    		\end{cases}
    		\]
    	\end{enumerate}
    \end{theorem}
    
    To prove Theorem \ref{clausenpb}, we need the following lemma.
    
    \begin{lemma}[{\cite[Lemma 14.8]{AIK}\label{vonlemma}}]
    	For $n$ and $a\in\Z_{\geq1}$, we have
    	\[
    	\stirs{n}{ap-1} \equiv \begin{cases}
    		\displaystyle \binom{c-1}{a-1} \bmod{p}& (if ~~n=a-1+c(p-1)~~for~some~~c\geq a), \\
    		0 \bmod{p} & (otherwise).
    	\end{cases}
    	\]
    \end{lemma}
	
	\section{Algebraic properties}
	
	\subsection{Explicit formulas and recurrence formulas}
	
	In this subsection, we give some explicit formulas and recurrence formulas of $\co{n}{k}$ and $\be{n}{k}.$
	
	\begin{proposition}\label{explicota}
		For $n\in\Z_{\geq0}$ and $k\in\Z$, we have
		\begin{align}
			\be{2n}{k}=\sum_{j=0}^{2n}\sum_{i=0}^{\lfloor j/2\rfloor} \frac{(-1)^j}{2^j}\frac{j!}{(2i+1)^k}\binom{j+1}{2i+1}\left(\frac{(j+1)(j+2)}{2}\stirs{2n}{j+2}+\stirs{2n+1}{j+1}\right).
		\end{align}
	\end{proposition}
	
	\begin{proof}
		We consider the generating function of $\be{n}{k}-\co{n}{k}$ instead of $\be{n}{k}$. By using Lemma \ref{tax2} and $\cosh{t}-1=2\sinh^2{\frac{t}{2}}$, we have
		\begin{align*}
			\sum_{n=0}^{\infty}\left(\be{n}{k}-\co{n}{k}\right)\frac{t^n}{n!}
			&=2\sum_{i=0}^{\infty}\frac{\ta{2i+2}}{(2i+1)^k}\\
			&=2\sum_{i=0}^{\infty}\frac{1}{(2i+1)^k}\sum_{n=2i+2}^{\infty}\sum_{j=2i+2}^{n}\frac{(-1)^jj!}{2^j}\binom{j-1}{2i+1}\stirs{n}{j}\frac{t^n}{n!}\\
			&=\sum_{n=2}^{\infty}\sum_{i=0}^{\lfloor(n-2)/2\rfloor}\sum_{j=2i+2}^{n}\frac{(-1)^jj!}{(2i+1)^k2^{j-1}}\binom{j-1}{2i+1}\stirs{n}{j}\frac{t^n}{n!}.
		\end{align*}
		Comparing coefficients on the both sides and using Lemma \ref{explicose}, we get
		\begin{align*}
			\be{2n}{k}&=\co{2n}{k}+\sum_{i=0}^{n-1}\sum_{j=2i+2}^{2n}\frac{(-1)^jj!}{(2i+1)^k2^{j-1}}\binom{j-1}{2i+1}\stirs{2n}{j}\\
			&=\sum_{j=0}^{2n}\sum_{i=0}^{\lfloor j/2\rfloor}\frac{(-1)^{j}j!}{(2i+1)^{k+1}2^j}\binom{j}{2i}\stirs{2n+1}{j+1}+\sum_{j=0}^{2n-2}\sum_{i=0}^{\lfloor j/2\rfloor}\frac{(-1)^j(j+2)!}{(2i+1)^{k+1}2^{j+1}}\binom{j}{2i}\stirs{2n}{j+2}.
		\end{align*}
	Since $\stirs{2n}{2n+1}=\stirs{2n}{2n+2}=0$, we have
	\[
    	\be{2n}{k}=\sum_{j=0}^{2n}\sum_{i=0}^{\lfloor j/2\rfloor}\frac{(-1)^{j}j!}{(2i+1)^{k+1}2^j}\binom{j}{2i}\stirs{2n+1}{j+1}+\sum_{j=0}^{2n}\sum_{i=0}^{\lfloor j/2\rfloor}\frac{(-1)^j(j+2)!}{(2i+1)^{k+1}2^{j+1}}\binom{j}{2i}\stirs{2n}{j+2}.
	\]
	Hence we obtain Proposition \ref{explicota}.
	\end{proof}

	\begin{proposition}\label{stircota}
	For $n, k\in\Z_{\geq1}$, we have 
	\begin{align*}
		\be{2n}{-k}&=\sum_{j=0}^{\min{\{2n,k-1\}}}\frac{j!(j+1)!}{2^{j+1}}\stirs{2n}{j}\stirs{k}{j+1}\\&~~~~+\sum_{j=0}^{\min{\{2n-1,k-1\}}}\frac{\{(j+1)!\}^2}{2^{j+1}}\stirs{2n}{j+1}\stirs{k}{j+1}\\
		&~~~~+\sum_{j=0}^{\min{\{2n-1,k-1\}}}\frac{(j+1)!(j+2)!}{2^{j+1}}\stirs{2n}{j+2}\stirs{k}{j+1}\\&~~~~+\sum_{j=0}^{\min{\{2n,k-1\}}}\frac{j!(j+1)!}{2^{j+1}}\stirs{2n+1}{j+1}\stirs{k}{j+1}.
	\end{align*}
\end{proposition}

\begin{proof}
	By using Lemma \ref{genfuncose}, we have
	\begin{align}
		\sum_{n=0}^{\infty}\sum_{k=1}^{\infty}\be{n}{-k}\frac{x^n}{n!}\frac{y^k}{k!}&=\cosh{x}\left\{\frac{e^x(e^y-1)}{1+e^x+e^y-e^{x+y}}+\frac{e^{-x}(e^y-1)}{1+e^{-x}+e^y-e^{-x+y}}\right\}\nonumber\\
		&=\frac{1}{4}\left\{\frac{e^y-1}{1-\frac{1}{2}(e^x-1)(e^y-1)}+\frac{e^y-1}{1-\frac{1}{2}(e^{-x}-1)(e^y-1)}\right\}\label{hitotume}\\
		&~~+\frac{1}{4}\left\{e^x\frac{(e^x-1)(e^y-1)}{1-\frac{1}{2}(e^x-1)(e^y-1)}+e^{-x}\frac{(e^{-x}-1)(e^y-1)}{1-\frac{1}{2}(e^{-x}-1)(e^y-1)}\right\}\label{futatume}\\
		&~~+\frac{1}{4}\left\{e^x\frac{e^y-1}{1-\frac{1}{2}(e^x-1)(e^y-1)}+e^{-x}\frac{e^y-1}{1-\frac{1}{2}(e^{-x}-1)(e^y-1)}\right\}\label{mittsume}.
	\end{align}
    For \eqref{hitotume}, by using Lemma \ref{expstir}, a calculation similar to that in \cite[Theorem 14.14]{AIK} shows that 
    \begin{align*}
    	\frac{1}{4}&\left\{\frac{e^y-1}{1-\frac{1}{2}(e^x-1)(e^y-1)}+\frac{e^y-1}{1-\frac{1}{2}(e^{-x}-1)(e^y-1)}\right\}\\
    	&=\frac{1}{4}\sum_{j=0}^{\infty}\frac{1}{2^j}(e^x-1)^j(e^y-1)^{j+1}+\frac{1}{4}\sum_{j=0}^{\infty}\frac{1}{2^j}(e^{-x}-1)^j(e^y-1)^{j+1}\\
    	&=\sum_{n,k=0}^{\infty}\left(\sum_{j=0}^{\min{\{2n,k\}}}\frac{j!(j+1)!}{2^{j+1}}\stirs{2n}{j}\stirs{k+1}{j+1}\right)\frac{x^{2n}}{(2n)!}\frac{y^{k+1}}{(k+1)!}.
    \end{align*}
	By the same calculations, we have
	\begin{align*}
		\sum_{n=0}^{\infty}\sum_{k=1}^{\infty}\be{n}{-k}\frac{x^n}{n!}\frac{y^k}{k!}&=\sum_{n,k=0}^{\infty}\left(\sum_{j=0}^{\min{\{2n,k\}}}\frac{j!(j+1)!}{2^{j+1}}\stirs{2n}{j}\stirs{k+1}{j+1}\right)\frac{x^{2n}}{(2n)!}\frac{y^{k+1}}{(k+1)!}\\
		&+\sum_{n,k=0}^{\infty}\left(\sum_{j=0}^{\min{\{2n+1,k\}}}\frac{\{(j+1)!\}^2}{2^{j+1}}\stirs{2n+2}{j+1}\stirs{k+1}{j+1}\right)\frac{x^{2n+2}}{(2n+2)!}\frac{y^{k+1}}{(k+1)!}\\
		&+\sum_{n,k=0}^{\infty}\left(\sum_{j=0}^{\min{\{2n+1,k\}}}\frac{(j+1)!(j+2)!}{2^{j+1}}\stirs{2n+2}{j+2}\stirs{k+1}{j+1}\right)\frac{x^{2n+2}}{(2n+2)!}\frac{y^{k+1}}{(k+1)!}\\
		&+\sum_{n,k=0}^{\infty}\left(\sum_{j=0}^{\min{\{2n,k\}}}\frac{j!(j+1)!}{2^{j+1}}\stirs{2n+1}{j+1}\stirs{k+1}{j+1}\right)\frac{x^{2n}}{(2n)!}\frac{y^{k+1}}{(k+1)!}.
	\end{align*}
	Comparing coefficients on the both sides, we obtain Proposition \ref{stircota}.
\end{proof}
    
    \begin{proposition}
    	For $k\in\Z$, and $m, n\in\Z_{\geq1}$ with $m\geq 2n$, we have
    	\begin{align}
    		\sum_{j=0}^{m}&(-1)^j\stirf{m+1}{j+1}\co{2n}{-k-j}=0\label{cose},\\
    		\sum_{j=0}^{m}&(-1)^j\stirf{m+1}{j+1}\be{2n}{-k-j}=0\label{cota}.
    	\end{align}
    \end{proposition}
    
    \begin{proof}
    	By Proposition \ref{explicose}, we have
    	\begin{align*}
    		\sum_{j=0}^{m}(-1)^j\stirf{m+1}{j+1}\co{2n}{-k-j}&=\sum_{l=0}^{n}(2l+1)^{k-2}\sum_{i=2m}^{2n}\frac{(-1)^{i+1}(i+1)!}{2^i}\binom{i}{2l}\stirs{2n+1}{i+1}\\
    		&\fift\times\sum_{j=0}^{m}(-1)^{j+1}\stirf{m+1}{j+1}(2l+1)^{j+1}.
    	\end{align*}
    	For the last sum, we have
    	\begin{align*}
    	\sum_{j=0}^{m}(-1)^{j+1}\stirf{m+1}{j+1}(2l+1)^{j+1}&=(-2l-1)(-2l)\cdots(-2l+m)\\
    	&=0.
    	\end{align*}
    	Hence we obtain \eqref{cose}. Also, by using \eqref{cotacose}, we have
    	\begin{align*}
    		\sum_{j=0}^{m}(-1)^j\stirf{m+1}{j+1}\be{2n}{-k-j}&=\sum_{l=0}^{n}\binom{2n}{2l}\sum_{j=0}^{m}(-1)^j\stirf{m+1}{j+1}\co{2l}{-k-j}=0.
    	\end{align*}
    	This completes the proof.
    \end{proof}
    
    \begin{remark}
    	By the same way, we have
    	\[
    	\sum_{j=0}^{m}(-1)^j\stirf{m+1}{j+1}B_{n}^{(-k-j)}=0~~(~k\in\Z, 0\leq n\leq m~).
    	\]
    \end{remark}

    \begin{example}
    	For $k=-2, n=2$ and $m=5$, we have $\co{4}{2}=\frac{176}{225}, \co{4}{1}=\frac{7}{15}, \co{4}{0}=0, \co{4}{-1}=1, \co{4}{-2}=16$ and $\co{4}{-3}=121$. Then we can see
    	\begin{align*}
    		\sum_{j=0}^{5}(-1)^j\stirf{6}{j+1}\co{4}{2-j}&=\frac{1408}{15}-\frac{1918}{15}+0-85+240-121\\
    		&=0.
    	\end{align*}
        Also, we have $\be{4}{2}=-\frac{199}{225}, \be{4}{1}=-\frac{8}{15}, \be{4}{0}=1, \be{4}{-1}=8, \be{4}{-2}=41$ and $\be{4}{-3}=200$. Then we can see
    \end{example}
        \begin{align*}
        	\sum_{j=0}^{5}(-1)^j\stirf{6}{j+1}\be{4}{2-j}&=-\frac{1592}{15}+\frac{2192}{15}+225-680+615-200\\
        	&=0.
        \end{align*}
    \subsection{Congruence formulas}
	
	In this subsection, we prove some congruence formulas.
	
		\begin{lemma}\label{congruencestir}
		Let $p$ be a prime number. For $n, m$ and $N\in\Z_{\geq1}$ with $n\equiv m\bmod\varphi(p^N)$ and $m,n\geq N$, and $j\in\Z_{\geq0}$, we have
		\[
		j!\stirs{n}{j}\equiv j!\stirs{m}{j} \bmod{p^N}.
		\]
	\end{lemma}

    \begin{proof}
    	By using Lemma \ref{bekistir} and the Euler's theorem, we have
    	\begin{align*}
    		j!\stirs{n}{j}&=\sum_{l=0}^{j}(-1)^{l+j}\binom{l}{j}l^n\\
    		&=\sum_{l=0}^{j}(-1)^{l+j}\binom{l}{j}l^m\\
    		&=j!\stirs{m}{j}.
    	\end{align*}
    Hence we obtain Lemma \ref{congruencestir}.
    \end{proof}

    The following theorem is the general form of Kummer type congruences for polycosecant numbers and polycotangent numbers, which is the generalization of Theorem \ref{kupa}.

    \begin{theorem}\label{kupoly}
    	Let $p$ be an odd prime number. For $n,m,k$ and $N\in\Z_{\geq 1}$ with $2m\equiv 2n \bmod \varphi(p^N)$ and $2n, 2m\geq N$, we have
    	\begin{align}
    		\co{2m}{-k}&\equiv\co{2n}{-k}\bmod{p^N}\label{kucose},\\
    		\be{2m}{-k}&\equiv\be{2n}{-k}\bmod{p^N}\label{kucota}.
    	\end{align}
    \end{theorem}
	
	\begin{proof}
		From Lemma \ref{congruencestir} and Proposition \ref{sasaki}, we have
		\begin{align*}
			\co{2n}{-k}&=\sum_{i=1}^{\min{\{2n+1,k\}}}\frac{i!(i-1)!}{2^{i-1}}\stirs{k}{i}\stirs{2n+1}{i}\\
			&\equiv \sum_{i=1}^{\min{\{2n+1,k\}}}\frac{i!(i-1)!}{2^{i-1}}\stirs{k}{i}\stirs{2m+1}{i} \bmod{p^N}.
		\end{align*}
		Since $\stirs{2m+1}{i}=0$ for $i=2m+2,\ldots, 2n+1$, the right hand side is equal to $\co{2m}{-k}$.
		Hence we obtain the formula \eqref{kucose}. By the same argument, we obtain the formula \eqref{kucota}.
	\end{proof}

    \begin{example}
    	For $p=3, m=2, n=5, k=3$ and $N=2$, we have $\co{4}{-3}=121$ and $\co{10}{-3}=88573$. Then we can see
    	\[
    	\co{4}{-3}\equiv \co{10}{-3}\equiv 4\bmod9.
    	\]
    	Also, we have $\be{4}{-3}=200$ and $\be{10}{-3}=786944$. Then we can see
    	\[
    	\be{4}{-3}\equiv\be{10}{-3}\equiv2\bmod9.
    	\]
    \end{example}
	
	\begin{remark}
		Since 
		\begin{align}
			\co{2n}{1}&=(2-2^{2n})B_{2n}, \label{coseber}\\
			\be{2n}{1}&=2^{2n}B_{2n},
		\end{align}
		we have
		\begin{align*}
			(1-p^{2n-1})\frac{\co{2n}{1}}{2n}&\equiv(1-p^{2m-1})\frac{\co{2m}{1}}{2m} \bmod{p^N},\\
			(1-p^{2n-1})\frac{\be{2n}{1}}{2n}&\equiv(1-p^{2m-1})\frac{\be{2m}{1}}{2m} \bmod{p^N}
		\end{align*}
		for $n, m, N\in\Z_{\geq0}$ with $2n\equiv 2m\bmod {(p-1)p^{N-1}}$ and $(p-1)\nmid 2n$.
	\end{remark}
	
	Moreover, we obtain more accurate evaluations of 2-order of $\co{2n}{-2k}$ and $\be{2n}{-2k-1}$. Note that some of the following results were suggested by private communication from Kaneko.
	
	\begin{proposition}\label{twoorder}
		For $n$ and $k\in\Z_{\geq1}$, we have
		\[
		\co{2n}{-2k}\equiv0\bmod{2^{2n}}.
		\]
	\end{proposition}
	
	\begin{proof}
		We prove this result by induction on $k$. For $k$=1, we can obtain that $\co{2n}{-2}=2^{2n}$ by induction on $n$. We assume that $\co{2n}{-2j}\equiv0\pmod{2^{2n}}$ holds for 
		$j\leq k-1$. Then we have
		\begin{align*}
			\co{2n}{-2k}&=\sum_{i=0}^{n}\binom{2n+1}{2i+1}\co{2(n-i)}{-2k+1}\\
			&=\sum_{i=0}^{n}\binom{2n+1}{2i+1}\sum_{j=0}^{n-i}\binom{2(n-i)+1}{2j+1}\co{2(n-i-j)}{-2(k-1)}\\
			&=\sum_{N=0}^{n}\sum_{i=0}^{n-N}\binom{2n+1}{2i+1}\binom{2(n-i)+1}{2N}\co{2N}{-2(k-1)}.
		\end{align*}
		Also, a direct calculation shows 
		\[
		\sum_{i=0}^{n-N}\binom{2n+1}{2i+1}\binom{2(n-i)+1}{2N}\equiv0 \bmod{2^{2n-2N}}.
		\]
		By combining these and the induction hypothesis, we obtain Proposition \ref{twoorder}.
	\end{proof}
	
	\begin{corollary}\label{cortwo}
		For $n\in\Z_{\geq1}$ and $k\in\Z_{\geq0}$, we have
		\[
		\be{2n}{-2k-1} \equiv0 \bmod{2^{2n-1}}.
		\]
	\end{corollary}
	
	\begin{proof}
		By using the formula \eqref{cotacose}, we have
		\begin{align*}
			\be{2n}{-2k-1}&=\sum_{i=0}^{n}\binom{2n}{2i}\co{2i}{-2k-1}\\
			&=\sum_{i=0}^{n}\sum_{j=0}^{i}\binom{2n}{2(n-i)}\binom{2i+1}{2(i-j)}\co{2(i-j)}{-2k}.
		\end{align*}	
		Also, a direct calculation shows
		\begin{align*}
			\sum_{j=0}^{i}\binom{2n}{2(n-i)}\binom{2i+1}{2(i-j)}&\equiv0 \bmod{2^{2(n-i+j)-1}}.
		\end{align*}
		By combining these and Proposition \ref{twoorder}, we obtain Corollary \ref{cortwo}.
	\end{proof}
	
	\begin{remark}
		When we put $N=1$, we can see that the polycosecant number and the polycotangent number have a period $p-1$ modulo $p$. In \cite[Theorem 4.3]{OS} and \cite[Theorem 6.5 and 6.6]{Sa}, Sakata showed the following proposition.
	\end{remark}
	
	\begin{proposition}\label{periodbc}
		Let $p$ be an odd prime number. For $k\in\Z_{\geq0}$, we have
		\begin{align}
			B_{p-1}^{(-k)} &\equiv \begin{cases}
				1 \bmod{p}& (k=0~\mm{or}~p-1 \nmid k), \\
				2 \bmod{p} & (k\neq0~\mm{and}~p-1 \mid k),
			\end{cases}\label{periodb}\\
			C_{p-2}^{(-k-1)} &\equiv \begin{cases}
				0 \bmod{p}& (p-1 \nmid k), \\
				1 \bmod{p} & (p-1 \mid k).
			\end{cases}\label{periodc}
		\end{align}
	\end{proposition}
	
	By duality formulas \eqref{dualb} and \eqref{dualc}, congruences in Proposition \ref{periodbc} can be written as 
	\begin{align*}
	    	B_{k}^{(-p+1)}\equiv \begin{cases}
	    		1 \bmod{p}& (k=0~\mm{or}~p-1 \nmid k), \\
	    		2 \bmod{p} & (k\neq0~\mm{and}~p-1 \mid k),
	    	\end{cases}~~~
	    	C_{k}^{(-p+1)}\equiv \begin{cases}
	    		0 \bmod{p}& (p-1 \nmid k), \\
	    		1 \bmod{p} & (p-1 \mid k).
	    	\end{cases}
	\end{align*}

    Also, Sakata and Pallewatta proved the following proposition in \cite{Sa} and in \cite{Pa}, respectively.
	
	\begin{proposition}\label{lemma}
		Let $p$ be an odd prime number. For $k\in\Z_{\geq0}$, we have
		\begin{align*}
			C_{p-1}^{(-k-1)}&\equiv 1 \bmod p,\\
			\co{p-1}{-2k-1}&\equiv 1 \bmod p. 
		\end{align*}
	\end{proposition}
	
	\begin{remark}
		By using Theorem \ref{kupoly}, we can see that $\co{p-1}{-k}\equiv 1\bmod p$ for $k\in\Z_{\geq1}$.
	\end{remark}
	
	By duality formulas \eqref{dualc} and \eqref{dualcose}, congruences in Proposition \ref{lemma} can be written as
	\begin{align*}
		C_{k}^{(-p)}\equiv 1 \bmod p,~~~
		\co{2k}{-p}\equiv 1 \bmod p. 
	\end{align*}
	We show that congruence formulas similar to \eqref{periodb} and \eqref{periodc} hold for $\co{n}{-k}$ and $\be{n}{-k}$.
	
	\begin{proposition}\label{sakatatype}
		Let $p$ be an odd prime number. For $n\in\Z_{\geq0}$, we have
		\begin{align*}
			\co{2n}{-p+1} &\equiv \begin{cases}
				0 \bmod{p}& (p-1 \nmid 2n), \\
				1 \bmod{p} & (p-1 \mid 2n),
			\end{cases}
			\\
			\be{2n}{-p+1} &\equiv \begin{cases}
				1 \bmod{p}& (2n=0~\mm{or}~ p-1 \nmid 2n), \\
				2 \bmod{p} & (2n\neq0~\mm{and}~p-1 \mid 2n).
			\end{cases}
		\end{align*}
	\end{proposition}
	
	\begin{proof}
		First, we show the case of $\co{2n}{-p+1}$. For $p=3$, we have $\co{2n}{-2}=4^n$, so the proposition holds. For $p\neq3$, we prove the proposition by induction on $n$. When $n=0$, we have $\co{0}{-p+1}=1$. When $n=1$, by using Proposition \ref{lemma}, we have
		\begin{align*}
			3\co{2}{-p+1}&=\co{2}{-p}-\co{0}{-p+1}\\
			&\equiv 1-1=0\bmod{p}
		\end{align*}
		(because $\co{2}{-p}\equiv 1 \bmod{p}$ was shown in \cite[Theorem 3.13]{Pa}), and we have $\co{2}{-p+1}\equiv 0\bmod p$. The same calculation shows that $\co{2n}{-p+1}\equiv 0\bmod p$ holds for $2n=2, 4, \ldots, p-3$. When $2n=p-1$, we have
		\begin{align*}
			\co{p-1}{-p+1}=\sum_{i=1}^{p-1}\frac{i!(i-1)!}{2^{i-1}}\stirs{p-1}{i}\stirs{p}{i}.
		\end{align*}
		Since $\stirs{p}{m}\equiv 0 \bmod p$ holds for $m=2,3,\ldots, p-1$, only the term $i=1$ remains and we have $\co{p-1}{-p+1}\equiv1 \bmod{p}$. For $\be{2n}{-p+1}$, since
		\begin{align*}
			\be{2n}{-p+1}=\sum_{i=0}^{n}\binom{2n}{2i}\co{2n-2i}{-p+1},
		\end{align*} 
		we have
		\begin{align*}
			\be{2n}{-p+1}&\equiv \co{0}{-p+1} =1\bmod p~~(2n\neq p-1),\\
			\be{p-1}{-p+1}&\equiv \co{0}{-p+1}+\co{p-1}{-p+1}\equiv 2\bmod p.
		\end{align*}
		Hence we obtain Proposition \ref{sakatatype}.
	\end{proof}

    In addition, $\co{n}{-k}$ and $\be{n}{-k}$ satisfy the congruence formula similar to \eqref{sum}.
    Here, we denote by $T_{2n+1}$ the tangent number defined by 
    \begin{align*}
    	\tan{t}=\sum_{n=0}^{\infty} T_{2n+1} \frac{t^{2n+1}}{(2n+1)!}
    \end{align*}
    (see \cite{No}), and 
    \[
    \widetilde{T}_{2n}=\begin{cases}
    	1&(n=0),\\
    	(-1)^{n-1}T_{2n+1}&(n\in\Z_{\geq1}).
    \end{cases}
    \]
    It is known that 
    \[
    T_{2n+1}=\begin{cases}
    	2^{2n}(2^{2n}-1)\frac{B_{2n}}{2n}&(n\in\Z_{\geq1}),\\
    	1&(n=0).
    \end{cases}
    \]
    Also, we have
    \[
    \sum_{n=0}^{\infty} \widetilde{T}_{2n} \frac{t^{2n}}{(2n)!}=1+\tanh^2{t}.
    \]
    
    \begin{proposition}
    	Let $p$ be an odd prime number. For $n \in\Z_{\geq 0}$ and $k, N\in\Z_{\geq 1}$ with $k\geq N$, we have
    	\begin{align}
    		2^{2n}\sum_{i=0}^{\varphi(p^N)-1} \co{2n}{-k-i}&\equiv (-1)^nT_{2n+1}\varphi(p^N) \bmod{p^N},\label{sumcose}\\
    		2^{2n}\sum_{i=0}^{\varphi(p^N)-1} \be{2n}{-k-i}&\equiv \widetilde{T}_{2n}\varphi(p^N) \bmod{p^N}.\label{sumcota}
    	\end{align}
    	Hence we have 
    	\begin{align*}
    		\sum_{i=0}^{\varphi(p^N)-1} \co{2n}{-k-i}&\equiv 0 \bmod{p^{N-1}},\\
    		\sum_{i=0}^{\varphi(p^N)-1} \be{2n}{-k-i}&\equiv 0 \bmod{p^{N-1}}.
    	\end{align*}
    \end{proposition}
    
    \begin{proof}
    	First, we show \eqref{sumcose}. By using Proposition \ref{sasaki} and Lemma \ref{bekistir}, we have
    	\begin{align*}
    		2^{2n}\sum_{i=0}^{\varphi(p^N)-1} \co{2n}{-k-i}&=\sum_{j=0}^{2n} 2^{2n-j}j!\stirs{2n+1}{i+1}\sum_{l=0}^{j+1}(-1)^{l+j+1}\binom{j+1}{l}l^{k+i}\\
    		&\equiv \varphi(p^N) \sum_{j=0}^{2n} (-1)^j2^{2n-j}(j+1)!\stirs{2n+1}{j+1} \bmod{p^N}.
    	\end{align*}
    	Also, a direct calculation shows that 
    	\begin{align*}
    		\sum_{n=0}^{\infty} (-1)^nT_{2n+1}\frac{t^{2n+1}}{(2n+1)!}&=\tanh{t}=\frac{(1-e^{2t})/2}{1-(1-e^{2t})/2}\\
    		&=\sum_{n=1}^{\infty} \left(\sum_{j=1}^{n} (-1)^{j-1}2^{n-j}j!\stirs{n}{j}\right)\frac{t^n}{n!}.
    	\end{align*}
    	Hence we have
    	\begin{align*}
    		(-1)^nT_{2n+1}=\sum_{j=0}^{2n} (-1)^{j}2^{2n-j}(j+1)!\stirs{2n+1}{j+1}.
    	\end{align*}
    	Therefore we obtain \eqref{sumcose}. For \eqref{sumcota}, we have
    	\begin{align*}
    		2^{2n}\sum_{i=0}^{\varphi(p^N)-1} \be{2n}{-k-i}
    		&=\sum_{j=0}^{n}2^{2n-2j}\binom{2n}{2j} 2^{2j}\sum_{i=0}^{\varphi(p^N)-1}\co{2j}{-k-i}\\
    		&\equiv \varphi(p^N)\sum_{j=0}^{n}2^{2n-2j}\binom{2n}{2j} (-1)^jT_{2j+1} \bmod{p^N}.
    	\end{align*}
    	Also, a direct calculation shows that
    	\begin{align*}
    		1+\tanh^2{t}&=\cosh{2t}\dv{}{t}\tanh{t}\\
    		&=\sum_{n=0}^{\infty} \left(\sum_{j=0}^{n} 2^{2n-2j}\binom{2n}{2j}(-1)^jT_{2j+1}\right)\frac{t^{2n}}{(2n)!}.
    	\end{align*}
    	Hence we have
    	\begin{align*}
    		\widetilde{T}_{2n}=\sum_{j=0}^{n} 2^{2n-2j}\binom{2n}{2j}(-1)^jT_{2j+1}.
    	\end{align*}
    	Therefore we obtain \eqref{sumcota}.
    \end{proof}
    
    \begin{remark}
    	The similar calculation shows that 
    	\begin{align*}
    		\sum_{i=0}^{\varphi(p^N)-1} C_n^{(-k-i)}&\equiv (-1)^n\varphi(p^N)\bmod{p^N}
    	\end{align*}
    	for an odd prime number $p$, $n \in\Z_{\geq 0}$ and $k, N\in\Z_{\geq 1}$ with $k\geq N$.
    \end{remark}

    \begin{example}
    	For $p=2, n=3, k=3$ and $N=2$, we have $\co{6}{-3}=1093, \co{6}{-4}=12160, \co{6}{-5}=111721, \co{6}{-6}=927424, \co{6}{-7}=7256173, \be{6}{-8}=54726400$ and $T_{7}=272$. Then we can see
    	\[
    	2^6\sum_{i=0}^{5}\co{6}{-3-i}\equiv 6\equiv -272\times6\bmod9.
    	\]
    	
    	Also, we have $\be{6}{-3}=3104, \be{6}{-4}=23801, \be{6}{-5}=174752, \be{6}{-6}=1257125, \be{6}{-7}=8948384, \be{6}{-8}=63318641$ and $\widetilde{T}_{7}=272$. Then we can see
    	\[
    	2^6\sum_{i=0}^{5}\be{6}{-3-i}\equiv 3\equiv 272\times6\bmod9.
    	\]
    \end{example}
	
	\subsection{Clausen von-Staudt type theorem}
	
	In this subsection, we prove the Clausen von-Staudt type theorem.
	
	\begin{proposition}\label{1von}
		Let $p$ be an odd prime number. For $n\in\Z_{\geq1}$, we have
		\[
		\ordp(d(2n))=\ordp(\widehat{\beta}(2n))=\ordp(b(2n)),
		\]
		where $b(2n)$, $d(2n)$ and $\widehat{\beta}(2n)$ denote the denominator of $B_{2n}$, $\co{2n}{1}$ and $\be{2n}{1}$ , respectively.
	\end{proposition}
	
	\begin{proof}
		By \eqref{coseber} and Theorem \ref{theoremvon}, we see that only a prime $p$ with $(p-1)\mid 2n$ possibly appears in the denominator of $\co{2n}{1}$. On the other hand, we have for this $p$,
		\[
		2-2^{2n}\equiv 1 \bmod p.
		\]
		Therefore we have $\ordp(d(2n))=\ordp(b(2n))$. For $\be{2n}{1}$, by considering the generating function of $\be{2n}{1}-\co{2n}{1}$, we have
		\[
		\be{2n}{1}-\co{2n}{1}=2(2^{2n}-1)B_{2n}.
		\]
		On the other hand, as stated above, if $(p-1) \nmid 2n$, then
		\[
		2^{2n}-1\equiv 0 \bmod p.
		\]
		Hence we have $\be{2n}{1}-\co{2n}{1}\in\zetp$ and obtain Proposition \ref{1von}.
	\end{proof}
	
	\begin{theorem}\label{vonco}
		Let $p$ be an odd prime number and $k\geq 2$ be an integer satisfying $k+2\leq p\leq 2n+1$.
		\begin{enumerate}
			\item When $(p-1) \mid 2n$, $p^k\co{2n}{k}$and $p^k\be{2n}{k}\in\zetp$. More explicitly,
			\[
			p^k\co{2n}{k}\equiv p^k\be{2n}{k}\equiv-1 \bmod p
			\]
			holds.
			\item When $(p-1) \nmid 2n$, $p^{k-1}\co{2n}{k}$and $p^{k-1}\be{2n}{k}\in\zetp$. More explicitly, 
			\begin{align*}
				p^{k-1}\co{2n}{k}&\equiv -\frac{1}{p}\stirs{2n}{p-1}+\sum_{\substack{l=p\\l\equiv1\bmod{p-1}}}^{2n}\binom{2n+1}{l}\sum_{j=0}^{\alpha-1}\frac{(-1)^jj!}{2^{j+1}}\stirs{\alpha}{j+1} \bmod{p}, \\
				p^{k-1}\be{2n}{k}&\equiv -\frac{1}{p}\stirs{2n}{p-1}+\sum_{\substack{l=p\\l\equiv1\bmod{p-1}}}^{2n}\binom{2n+1}{l}\sum_{j=0}^{\alpha-1}\frac{(-1)^jj!}{2^{j+1}}\stirs{\alpha}{j+1}\\
				&~~~+\sum_{j=p}^{\gamma} \frac{(-1)^j}{2^j}\frac{(j+2)!}{p}\stirs{2n}{j+2}\binom{j+1}{p} \bmod p
			\end{align*}
			hold, where $\alpha\in\Z$ is the remainder of $2n$ modulo $p-1$ and $\gamma=\min{\left\{2n, 2p-3\right\}}$.
		\end{enumerate}
	\end{theorem}
	
	\begin{proof}
		First, we see Theorem \ref{vonco} for $\co{n}{k}$. For simplicity, we put
		\[
		\dnkij=\frac{(-1)^j}{2^j}\frac{j!}{(2i+1)^k}\stirs{2n+1}{j+1}\binom{j+1}{2i+1}~~(~0\leq j\leq 2n, 0\leq i\leq\lfloor j/2\rfloor~).
		\] 
		By using this, we can write
		\begin{align}\label{codnkij}
		\co{n}{k}=\sum_{j=0}^{2n}\sum_{i=0}^{\lfloor j/2\rfloor} \dnkij.
		\end{align}
		Furthermore, we write $2i+1$ as $ap^e~(\gcd(a,p)=1, e\geq0)$. Note that 
		\[
		\ordp(\dnkij)\geq\ordp\left(\frac{j!}{(2i+1)^k}\right).
		\] 
		First, we assume $e=0$. Then $\dnkij\in\zetp$ holds and $p^{k-1}\dnkij\equiv0 \bmod p$ by the assumption $k\geq 2$.
		
		Next, we assume $e\geq 2$. By the same argument in \cite[$\S$14]{AIK}, we have
		\begin{align*}
			\ordp\left(\frac{j!}{(2i+1)^k}\right)&\geq-k+1,
		\end{align*}
		so we have $p^{k-1}\dnkij\in\zetp$ and $p^k\dnkij\equiv0\bmod p$. We see that $p^{k-1}\dnkij\not\equiv 0\bmod p$ holds for only the case $e=2, 2i+1=p^2, p=k+2$, and $\ordp(j!)=\ordp((2i)!)$. Namely, the condition $2i+1=p^2$ and $\ordp(j!)=\ordp((2i)!)$ holds only if $j=2i$ holds. On the other hand, the assumption $(p-1)\nmid 2n$ and Lemma \ref{vonlemma} imply 
		\[
		\stirs{2n}{j}=\stirs{2n}{p^2-1}\equiv0 \bmod p~~(~\mm{changing}~n~\mm{by}~2n~\mm{and}~a=p~),
		\]
		so we have
		\[
		\stirs{2n+1}{j+1}\equiv 0 \bmod p,
		\]
		which implies that $p^{k-1}\dnkij\equiv0 \bmod p$ holds.
		
		Now, we suppose $e=1~($namely $2i+1=ap)$. If $a\geq 3$, $\ordp\left(\frac{j!}{(2i+1)^k}\right)>-k+1$ since $p^2 \mid (ap-1)!$. The case $a=2$ never happens since $p$ is an odd prime. If $a=1~($namely $2i+1=p)$, then we have
		\[
		\dnkij=\frac{(-1)^jj!}{2^jp^k}\stirs{2n+1}{j+1}\binom{j+1}{p}.
		\]  
		If $j>2i$ holds, we have $\ordp(\dnkij)\geq -k+1$. If $j=2i$ holds, then we have
		\[
		\dnkij=\frac{(p-1)!}{2^{p-1}p^k}\stirs{2n+1}{p}.
		\]
		By using Lemma \ref{vonlemma} and the formula $\stirs{2n+1}{p}=\stirs{2n}{p-1}+p\stirs{2n}{p}$, we get the following results.
		\begin{enumerate}
			\item If $(p-1) \nmid 2n$, then we have $\stirs{2n+1}{p}\equiv 0 \bmod p$ and
			\[
			p^{k-1}\dnkij\equiv -\frac{1}{p}\stirs{2n+1}{p} \bmod p.
			\] 
			\item If $(p-1) \mid 2n$, then we have $\stirs{2n}{p-1}\equiv 1 \bmod p$ and
			\[
			p^k\dnkij\equiv -1 \bmod p.
			\]
		\end{enumerate}
		Hence by \eqref{codnkij}, we have the following.
		\begin{enumerate}
			\item If $(p-1) \mid 2n$, then we have
			\[
			p^{k}\co{2n}{k}\equiv -1 \bmod p.
			\] 
			\item If $(p-1) \nmid 2n$, then we have
			\[
			p^{k-1}\co{2n}{k}\equiv -\frac{1}{p}\stirs{2n+1}{p}+\sum_{j=p}^{2n} \frac{(-1)^j}{2^j}\frac{j!}{p}\stirs{2n+1}{j+1}\binom{j+1}{p} \bmod p.
			\]
		\end{enumerate}
		For the former term, we have
		\begin{align*}
			-\frac{1}{p}\stirs{2n+1}{p}&=-\frac{1}{p}\stirs{2n}{p-1}-\stirs{2n-1}{p-1}-p\stirs{2n-1}{p}\\
			&\equiv -\frac{1}{p}\stirs{2n}{p-1}\bmod p
		\end{align*} 
		by Lemma \ref{vonlemma}. For the latter sum, by using the formula 
		\[
		\stirs{n}{j+1}\binom{j+1}{p}=\sum_{k}^{}\stirs{k}{p}\stirs{n-k}{j+1-p}\binom{n}{k}
		\]
		in \cite{GKP}, we have
		\begin{align*}
			\sum_{j=p}^{2n} \frac{(-1)^jj!}{2^jp}\stirs{2n+1}{j+1}\binom{j+1}{p}&=\sum_{j=p}^{2n} \frac{(-1)^jj!}{2^jp}\sum_{l=j}^{2n}\stirs{l}{p}\stirs{2n+1-l}{j+1-p}\binom{2n+1}{l}.
	    \end{align*}
        Also, we have
        \begin{align*}
        	\stirs{l}{p}=\stirs{l-1}{p-1}+p\stirs{l-1}{p}&\equiv \begin{cases}
        		1\bmod{p}~(l-1\equiv0 \bmod{p-1}),\\
        		0\bmod{p}~(otherwise),
        	\end{cases}\\
            \stirs{2n+1-l}{j+1-p}&\equiv\stirs{2n+1-l}{j}\bmod{p}.
        \end{align*}
        Hence we have
        \begin{align*}
			\sum_{j=p}^{2n} \frac{(-1)^jj!}{2^jp}\sum_{l=j}^{2n}\stirs{l}{p}\stirs{2n+1-l}{j+1-p}\binom{2n+1}{l}&\equiv \sum_{\substack{l=p\\l\equiv1\bmod{p-1}}}^{2n}\binom{2n+1}{l}\sum_{j=0}^{l-p}\frac{(-1)^jj!}{2^{j+1}}\stirs{2n+1-l}{j+1}\\
			&\equiv \sum_{\substack{l=p\\l\equiv1\bmod{p-1}}}^{2n}\binom{2n+1}{l}\sum_{j=0}^{\alpha-1}\frac{(-1)^jj!}{2^{j+1}}\stirs{\alpha}{j+1} \bmod p.
		\end{align*}
		Hence we obtain the desired result for $\co{n}{k}$.
		For $\be{n}{k}$, it is sufficient to consider the case $e=1, a=1$ and $j=2i=p-1$. In this case, we have
		\[
		p^k\frac{(-1)^{p-1}j!}{2^{p-1}p^k}\frac{p(p+1)}{2}\stirs{2n}{p+1}\equiv0 \bmod p,
		\]
		so 
		\[
		p^k\be{2n}{k}\equiv p^k\co{2n}{k}\bmod p
		\]
		holds, and we obtain Theorem \ref{vonco}.
	\end{proof}
	
	\begin{remark}
		By using Proposition \ref{explicose} and Proposition \ref{explicota}, we have
		\[
		\ordp(\co{2n}{k}), ~~\ordp(\be{2n}{k}) \geq 0
		\]
		for an odd prime number $p$ with $2n+2\leq p$.
		Also, by using Proposition \ref{recurcose}, the formula \eqref{cotacose} and the fact that $\co{0}{k}=1$ for all $k\in\Z$, we have
		\[
		{\mm{ord}}_2(\co{2n}{k}), ~~{\mm{ord}}_2(\be{2n}{k}) \geq 0.
		\]
	\end{remark} 	
 
    \section{On symmetrized polycosecant numbers.}
    
    In \cite{KST}, symmetrized poly-Bernoulli numbers are defined by Kaneko, Sakurai and Tsumura as follows:
    
    \begin{definition}[{\cite[Section2]{KST}}]
    	For $l, m$ and $n\in\Z_{\geq0}$, symmetrized poly-Bernoulli numbers $\{\mathscr{B}_m^{(-l)}(n)\}$ are defined by 
    	\[
    	\mathscr{B}_m^{(-l)}(n)=\sum_{j=0}^{\infty} \stirf{n}{j}B_{m}^{(-l-j)}(n),
    	\]
    	where $B_{m}^{(k)}(x)$ is the poly-Bernoulli polynomial defined by
    	\[
    	e^{-xt}\frac{\mm{Li}_k(1-e^{-t})}{1-e^{-t}}=\sum_{n=0}^{\infty}B_{n}^{(k)}(x)\frac{t^n}{n!}~(k\in\Z).
    	\]
    \end{definition}

    It can be checked that
    \begin{align}
    \mathscr{B}_m^{(-l)}(0)=B_m^{(-l)},~~~\mathscr{B}_m^{(-l)}(1)=C_m^{(-l-1)}. \label{casesymber}
    \end{align}
    In particular, $\mathscr{B}_m^{(-l)}(n)$ satisfies the following duality formula.
    
    \begin{theorem}[{\cite[Corollary 2.2]{KST}\label{symber}}]
    	For $l, m$ and $n\in\Z_{\geq0}$, we have
    	\[
    	\mathscr{B}_m^{(-l)}(n)=\mathscr{B}_l^{(-m)}(n).
    	\]
    \end{theorem}

    Theorem \ref{symber} can be regarded as a generalization of duality formulas for $B_m^{(-l)}$ and $C_m^{(-l-1)}$ because of \eqref{casesymber}. This theorem can be shown by considering the following generating function in two variables.
    
    \begin{theorem}[{\cite[Theorem 2.1]{KST}}]
    	For $n\in\Z_{\geq0}$, we have
    	\[
    	\sum_{l=0}^{\infty}\sum_{m=0}^{\infty} \mathscr{B}_m^{(-l)}(n)\frac{x^l}{l!}\frac{y^m}{m!}=\frac{n!e^{x+y}}{(e^x+e^y-e^{x+y})^{n+1}}.
    	\]
    \end{theorem} 

    The following proposition says that $\mathscr{B}_m^{(-l)}(n)\in\Z_{\geq1}$.
    
    \begin{proposition}[{\cite{KST}}]
    	For $l, m$ and $n\in\Z_{\geq0}$, we have
    	\[
    	\mathscr{B}_m^{(-l)}(n)=\sum_{j=0}^{\min{(l, m)}} n!(j!)^2\binom{j+n}{n}\stirs{l+1}{j+1}\stirs{m+1}{j+1}. 
    	\]
    \end{proposition}

    Also, there are some studies on combinatorial properties of $\mathscr{B}_m^{(-l)}(n)$ (see \cite{BM}). In this section, we define symmetrized polycosecant numbers $\{\cosym{m}{-k}{n}\}$ and prove duality formulas for them.
    
    \begin{definition}
    	For $l, m$ and $n\in\Z_{\geq0}$, we define ~$\copoly{}{m}{-l}{n}$ by
    	\begin{align}
    		\frac{1}{2}(e^t+1)^{1-n}\frac{{\mm{Li}}_{-l}(\ta{})}{\sinh{t}}+\frac{1}{2}(e^{-t}+1)^{1-n}\frac{{\mm{Li}}_{-l}(-\ta{})}{\sinh{(-t)}}&=\sum_{m=0}^{\infty}\copoly{}{m}{-l}{n}\frac{t^m}{m!}
    	\end{align}
    	and symmetrized polycosecant numbers $\cosym{m}{-l}{n}$ by
    	\[
    	\cosym{m}{-l}{n}=\sum_{j=0}^{n}\stirf{n}{j}\copoly{}{m}{-l-j}{n}.
    	\]
    \end{definition}

    	Since the generating function is an even function, we have $\cosym{2m+1}{-l}{n}=0$ for $l, m$ and $n\in\Z_{\geq0}$. Also, when $n=1$, $\cosym{m}{-l}{1}=\frac{1}{2}\co{m}{-l-1}$. Hence we can regard $\cosym{m}{-l}{n}$ as a generalization of $\co{m}{-l}$. The duality formula for them is as follows.
    	
    \begin{theorem}\label{cosym}
    	For $l, m$ and $n\in\Z_{\geq0}$, we have
    	\[
    	\cosym{2m}{-2l}{n}=\cosym{2l}{-2m}{n}.
    	\]
    \end{theorem}

    To prove this theorem, we need the following lemmas. For simplicity, we put
\begin{align*}
	f_{1,n}(t,y)&=\frac{n!e^{t+y}}{(1+e^t+e^y-e^{t+y})^{n+1}},\\
	f_{2,n}(t,y)&=\frac{n!e^{-t+y}}{(1+e^{-t}+e^y-e^{-t+y})^{n+1}}.\\
\end{align*}

\begin{lemma}\label{fncosym}
	For $n\in\Z_{\geq0}$, we have
	\begin{align*}
		f_{1,n}(t,y)&=(e^t+1)^{-n}\sum_{j=0}^{n}\stirf{n}{j}\pdv{j}{y}f_{1,0}(t,y),\\
		f_{2,n}(t,y)&=(e^{-t}+1)^{-n}\sum_{j=0}^{n}\stirf{n}{j}\pdv{j}{y}f_{2,0}(t,y).\\
	\end{align*}
\end{lemma}

\begin{lemma}\label{gencosym}
	For $n\in\Z_{\geq0}$, we have
	\begin{align*}
		f_{1,n}(t,y)+f_{2,n}(t,y)=\sum_{m=0}^{\infty}\sum_{l=0}^{\infty}\copasym{}{m}{-l}{n}\frac{t^m}{m!}\frac{y^l}{l!}.
	\end{align*}
\end{lemma}

From Lemmas \ref{fncosym} and \ref{gencosym}, and by $\cosym{2m+1}{-l}{n}=0$, we have

\begin{align*}
	2\sum_{m=0}^{\infty}\sum_{l=0}^{\infty}\cosym{2m}{-2l}{n}\frac{t^{2m}}{(2m)!}\frac{y^{2l}}{(2l)!}&=f_{1,n}(t,y)+f_{2,n}(t,y)+f_{1,n}(t,-y)+f_{2,n}(t,-y)\\
	&=\frac{n!e^{t+y}}{(1+e^t+e^y-e^{t+y})^{n+1}}+\frac{n!e^{-t+y}}{(1+e^{-t}+e^y-e^{-t+y})^{n+1}}\\
	&\five+\frac{n!e^{t-y}}{(1+e^t+e^{-y}-e^{t-y})^{n+1}}+\frac{n!e^{-t-y}}{(1+e^{-t}+e^{-y}-e^{-t-y})^{n+1}}.
\end{align*}

Hence we obtain Theorem \ref{cosym}.

\begin{proof}[Proof of Lemma \ref{fncosym}]
	We prove the result $f_{1,n}(t,y)$ by induction on $n$. The case $n=0$ is straightforward.
	
	For $n\geq1$, a direct calculation shows that
	\begin{align*}
		\pdv{}{y}f_{1,n}(t,y)&=(e^t+1)f_{1,n+1}(t,y)-nf_{1,n}(t,y).
	\end{align*}
    By the induction hypothesis, we have
	\begin{align*}
		f_{1,n+1}(t,y)&=(e^t+1)^{-1}n!e^{-ny}e^t\pdv{}{y}\frac{1}{(1+e^{-y}+e^{t-y}-e^t)^{n+1}}\\
		&=\frac{(n+1)!e^{t+y}}{(1+e^t+e^y-e^{t+y})^{n+1}}.
	\end{align*}
    By the similar argument, we obtain the result for $f_{2,n}(t,y)$.
\end{proof}

\begin{proof}[Proof of Lemma \ref{gencosym}]
	Since
	\begin{align*}
		\sum_{m=0}^{\infty}\sum_{l=0}^{\infty}\sum_{j=0}^{n}&\stirf{n}{j}\copoly{}{m}{-l-j}{n}\frac{t^m}{m!}\frac{y^l}{l!}\\
		=&\frac{1}{2}\sum_{j=0}^{n}\stirf{n}{j}\sum_{l=0}^{\infty}(e^t+1)^{1-n}\frac{{\mm{Li}}_{-l-j}(\ta{})}{\sinh{t}}\frac{y^l}{l!}\\
		&+\frac{1}{2}\sum_{j=0}^{n}\stirf{n}{j}\sum_{l=0}^{\infty}(e^{-t}+1)^{1-n}\frac{{\mm{Li}}_{-l-j}(-\ta{})}{\sinh{(-t)}}\frac{y^l}{l!},
	\end{align*}
	we aim to prove that $f_{1,n}(t,y)$ coincides with the former sum and $f_{2,n}(t,y)$ coincides with the latter sum. In particular, since the argument for $f_{2,n}(t,y)$ is the same as that for $f_{1,n}(t,y)$, we show the result for $f_{1,n}(t,y)$. When $n=0$, we have
	\begin{align}
	    \sum_{l=0}^{\infty}{\mm{Li}}_{-l}(z)\frac{y^l}{l!}=\sum_{l=0}^{\infty}\sum_{m=0}^{\infty}m^lz^m\frac{y^l}{l!}=\frac{e^yz}{1-e^yz}~(|z|<1).\label{genli}
	\end{align}
    Hence we have
	\begin{align*}
		\sum_{l=0}^{\infty}\frac{1}{2}(e^t+1)\frac{\mm{Li}_{-l}(\ta{})}{\sinh{t}}\frac{y^l}{l!}&=\frac{1}{2}(e^t+1)\frac{1}{\sinh{t}}\frac{e^y\ta{}}{1-e^y\ta{}}\\
		&=\frac{e^{t+y}}{1+e^t+e^y-e^{t+y}}\\
		&=f_{1,0}(t,y)
	\end{align*}
	
	For $n\geq1$, by using \eqref{genli}, we have
	\begin{align*}
		&\frac{1}{2}\sum_{j=0}^{n}\stirf{n}{j}\sum_{l=0}^{\infty}(e^t+1)^{1-n}\frac{{\mm{Li}}_{-l-j}(\ta{})}{\sinh{t}}\frac{y^l}{l!}\\
		&=\frac{1}{2}\sum_{j=1}^{n}\stirf{n}{j}\pdv{j-1}{y}\sum_{k=0}^{\infty}(e^t+1)^{1-n}\pdv{}{t}{\mm{Li}}_{-k}(\ta{})\frac{y^k}{k!}\\
		&=\frac{1}{2}\sum_{j=1}^{n}\stirf{n}{j}\pdv{j-1}{y}(e^t+1)^{1-n}\times\pdv{}{t}\frac{e^y\ta{}}{1-e^y\ta{}}\\
		&=\sum_{j=1}^{n}\stirf{n}{j}\pdv{j-1}{y}(e^t+1)^{1-n}\frac{e^{t+y}}{(1+e^t+e^y-e^{t+y})^2}\\
		&=\sum_{j=1}^{n}\stirf{n}{j}\pdv{j}{y}(e^t+1)^{-n}\frac{e^{t+y}}{1+e^t+e^y-e^{t+y}}\\
		&=f_{1,n}(t,y).
	\end{align*}
\end{proof}

From Lemma \ref{gencosym}, we obtain an explicit formula of $\cosym{m}{-l}{n}$.

    \begin{proposition}\label{expsym}
	For $l, m$ and $n\in\Z_{\geq0}$, we have
	\[
	\cosym{2m}{-l}{n}=\frac{n!}{2^{n+1}}\sum_{j=0}^{\min{\{2m, l\}}} \frac{(j!)^2}{2^{j-1}} \binom{j+n}{n}\stirs{2m+1}{j+1}\stirs{l+1}{j+1}.
	\]
	Hence we obtain $\frac{2^{n+1}}{n!}\cosym{2m}{-l}{n}\in\Z_{\geq0}$.
\end{proposition}

    \begin{proof}
	Since we have
	\begin{align*}
		\left(\frac{1}{1-z}\right)^{n+1}=\sum_{j=0}^{\infty} \binom{j+n}{n}z^j,
	\end{align*}
    a direct calculation shows that
	\begin{align*}
		\frac{n!e^{t+y}}{(1+e^t+e^y-e^{t+y})^{n+1}}
		&=\frac{n!}{2^{n+1}}e^{t+y}\left(\frac{1}{1-\frac{1}{2}(e^t-1)(e^y-1)}\right)^{n+1}\\
		&=\frac{n!}{2^{n+1}}e^{t+y}\sum_{j=0}^{\infty} \binom{j+n}{n}\frac{1}{2^j}(e^t-1)^j(e^y-1)^j\\
		&=\frac{n!}{2^{n+1}}\sum_{m=j}^{\infty}\sum_{l=j}^{\infty}\left(\sum_{j=0}^{\min{\{m, l\}}} \frac{(j!)^2}{2^{j}}\binom{j+n}{n}\stirs{m+1}{j+1}\stirs{l+1}{j+1}\right)\frac{t^m}{m!}\frac{y^l}{l!}.
	\end{align*}
    By Lemma \ref{gencosym}, we obtain Proposition \ref{expsym}.
\end{proof}

\begin{example}
	When $n=1$, Theorem \ref{cosym} implies
	\[
	\co{2m}{-2l-1}=\co{2l}{-2m-1}~( l, m\in\Z_{\geq0} ).
	\]
	\end{example}

\begin{example}
	
	When $n=0$, from Propositions \ref{sasaki}, \ref{stircota} and \ref{expsym}, we have
	\begin{align*}
	1+\frac{e^{t+y}}{1+e^t+e^y-e^{t+y}}&+\frac{e^{-t+y}}{1+e^{-t}+e^y-e^{-t+y}}+\frac{e^{t-y}}{1+e^t+e^{-y}-e^{t-y}}+\frac{e^{-t-y}}{1+e^{-t}+e^{-y}-e^{-t-y}}\\
	&=\sum_{m=0}^{\infty}\sum_{l=0}^{\infty} \left(\be{2m}{-2l}+\co{2m}{-2l}+\co{2l}{-2m}\right)\frac{t^{2m}}{(2m)!}\frac{y^{2l}}{(2l)!}.
	\end{align*}
    Hence Theorem \ref{cosym} implies
    \begin{align*}
    	\be{2m}{-2l}+\co{2m}{-2l}+\co{2l}{-2m}&=\be{2l}{-2m}+\co{2l}{-2m}+\co{2m}{-2l}~(l,m\in\Z_{\geq 0}).
    \end{align*}
    Therefore we have
    \[
    \be{2m}{-2l}=\be{2l}{-2m}.
    \]
    \end{example}

\begin{example}
    
    When $n=2$, Theorem \ref{cosym} implies
    \[
    \sum_{j=0}^{2m} \binom{2m}{j}E_{j}(0)\left(\widetilde{D}_{2m-j}^{(-2l-1)}+\widetilde{D}_{2m-j}^{(-2l-2)}\right)=\sum_{j=0}^{2l} \binom{2l}{j}E_{j}(0)\left(\widetilde{D}_{2l-j}^{(-2m-1)}+\widetilde{D}_{2l-j}^{(-2m-2)}\right),
    \]
    where
    \[
    \frac{2e^{xt}}{e^t+1}=\sum_{m=0}^{\infty} E_m(x)\frac{t^m}{m!},\five
    \frac{{\mm{Li}}_{-k}(\ta{})}{\sinh{t}}=\sum_{m=0}^{\infty} \widetilde{D}_m^{(-k)}\frac{t^m}{m!}.
    \]
\end{example}

	\section*{Acknowledgments}
	The author would like to thank my supervisor Professor Hirofumi Tsumura for his kind advice and helpful comments. The author also thanks Professor Masanobu Kaneko for his important suggestions. This work was supported by JST, the establishment of university fellowships towards the creation of science technology innovation, Grant Number JPMJFS2139.


\begin{bibdiv}
		\begin{biblist}
			\bib{AIK}{book}{
				author={T. Arakawa},author={T. Ibukiyama},author={M. Kaneko. with appendix by Don Zagier},
				title={Bernoulli numbers and Zeta Functions},
				publisher={Springer Monographs in Mathematics, Springer, Tokyo},
				date={2014}
			}
			\bib{AK1}{article}{
				author={T. Arakawa},author={M. Kaneko},
				title={Multiple zeta values, poly-Bernoulli numbers, and related zeta functions},
				journal={Nagoya Math. J.},
				volume={153},
				date={1999},
				number={},
				pages={189--209},
				issn={},
			}
			\bib{BM}{article}{
				author={B. B\'enyi},author={T. Matsusaka},
				title={On the combinatorics of symmetrized poly-Bernoulli numbers},
				journal={Electron. J. Combin},
				volume={28},
				date={2021},
				number={1},
				pages={20pp},
				issn={},
			}
			\bib{GKP}{book}{
				author={R. L. Graham},author={D. E. Knuth},author={O. Patashnik},
				title={Concrete mathematics},
				publisher={Addison-Wesley},
				date={1994},
			}
			\bib{IKT}{article}{
				author={K. Imatomi},author={M. Kaneko},author={E. Takeda},
				title={Multi-poly-Bernoulli numbers and finite multiple zeta values},
				journal={J. Integer Seq.},
				volume={17},
				date={2014},
				number={4},
				pages={12pp},
				issn={},
			}
			\bib{K1}{article}{
				author={M. Kaneko},
				title={Poly-Bernoulli numbers},
				journal={J. Th\'eor. Nombres Bordeaux},
				volume={9},
				date={1997},
				number={1},
				pages={199--206},
				issn={},
			}
           \bib{KKT}{article}{
           author={M. Kaneko},author={Y. Komori},author={H. Tsumura},
           title={On Arakawa-Kaneko zeta-functions associated with $GL_2(\C)$ and their functional relations II},
           journal={in preparation},
           volume={},
           date={},
           number={},
           pages={},
           issn={},
           }
			\bib{KPT}{article}{
				author={M. Kaneko},author={M. Pallewatta},author= {H. Tsumura},
				title={On polycosecant numbers},
				journal={J. Integer Seq.},
				volume={23},
				date={2020},
				number={6},
				pages={17pp},
				issn={},
			}
		    \bib{KST}{article}{
		    	author={M. Kaneko},author={F. Sakurai},author={H. Tsumura},
		    	title={On a duality formula for certain sums of values of poly-Bernoulli polynomials and its application},
		    	journal={J. Th\'eor. Nombres Bordeaux},
		    	volume={30},
		    	date={2018},
		    	number={1},
		    	pages={203--218},
		    	issn={},
		    }
			\bib{Ka}{article}{
				author={Y. Katagiri},
				title={Kummer-type congruences for multi-poly-Bernoulli numbers},
				journal={Commentarii mathematici Universitatis Sancti Pauli},
				volume={69},
				date={2021},
				number={},
				pages={75--85},
				issn={},
			}
			\bib{Ki}{thesis}{
				author={R. Kitahara},
				title={On Kummer-type congruences for poly-Bernoulli numbers~(in Japanese)},
				language={},
				organization={Tohoku University, master thesis~(2012)},
				volume={},
				date={},
				number={},
				pages={},
				issn={},
			}
			\bib{No}{book}{
				author={N. E. N$\mm{\ddot{o}}$rlund},
				title={Vorlesungen $\ddot{u}$ber Differenzenrechnung},
				publisher={Springer Verlag},
				date={1924},
			}
			\bib{Pa}{thesis}{
				author={M. Pallewatta},
				title={On polycosecant numbers and level Two generalization of Arakawa-Kaneko zeta functions},
				language={},
				organization={Kyushu University, doctor thesis~(2020)},
				volume={},
				date={},
				number={},
				pages={52pp},
				issn={},
			}
		    \bib{OS}{article}{
		    	author={Y. Ohno},author={M. Sakata},
		    	title={On certain properties of poly-Bernoulli numbers with negative index},
		    	journal={J. School sci. Eng.},
		    	volume={49},
		    	date={2013},
		    	number={},
		    	pages={5--7},
		    	issn={},
		    }
			\bib{OS2}{article}{
				author={Y. Ohno},author={Y. Sasaki},
				title={Recursion formulas for poly-Bernoulli numbers and their applications},
				journal={Int. J. Number Theory},
				volume={17},
				date={2021},
				number={1},
				pages={175--189},
				issn={},
			}
			\bib{Sa}{thesis}{
				author={M. Sakata},
				title={On p-adic properties of poly-Bernoulli numbers~(in Japanese)},
				language={},
				organization={Kindai University, master thesis~(2012)},
				volume={},
				date={},
				number={},
				pages={},
				issn={},
			}
			\bib{Sasaki}{article}{
				author={Y. Sasaki},
				title={On generalized poly-Bernoulli numbers and related L-functions},
				journal={Journal of Number Theory},
				volume={132},
				date={2012},
				number={},
				pages={156--170},
				issn={},
			}
		\end{biblist}
	\end{bibdiv}
\end{document}